\documentclass[11pt, twoside]{article}
\usepackage{amssymb, amsmath, amsthm}
\usepackage[left=27mm, right=27mm, bottom=25mm, top=27mm]{geometry}
\usepackage{inputenc}
\usepackage{caption}
\usepackage{fancyhdr}
\usepackage{tikz}
\usetikzlibrary{positioning}
\usepackage{titling}
\usepackage{hyperref}
\predate{}
\postdate{}
\usepackage{lipsum}
\newtheorem{theorem}{Theorem}

\usepackage{titling}
\usepackage{xcolor}
\usepackage{placeins}
\usepackage{tikz}
\usepackage{verbatim}
\usepackage{titling}
\usepackage[T1]{fontenc}
\usepackage{currvita}
\usepackage{bbm}
\usepackage{enumitem}

\theoremstyle{definition}

\theoremstyle{remark}

\newcommand{\Sph}{\mathbb{S}}
\newcommand{\cU}{\mathcal{U}}
\newcommand{\R}{\mathbb{R}}
\begin{document}
\newcommand{\Addresses}{
\bigskip
\footnotesize

\medskip

\noindent David~Ellis, \textsc{School of Mathematics, University of Bristol, The Fry Building, Woodland Road, Bristol, BS8 1UG, UK.}\par\noindent\nopagebreak\textit{Email address: }\texttt{david.ellis@bristol.ac.uk}

\medskip

\noindent Maria-Romina~Ivan, \textsc{Department of Pure Mathematics and Mathematical Statistics, Centre for Mathematical Sciences, Wilberforce Road, Cambridge, CB3 0WB, UK.}\par\noindent\nopagebreak\textit{Email addresses: }\texttt{mri25@dpmms.cam.ac.uk}

\medskip

\noindent Imre~Leader, \textsc{Department of Pure Mathematics and Mathematical Statistics, Centre for Mathematical Sciences, Wilberforce Road, Cambridge, CB3 0WB, UK.}\par\noindent\nopagebreak\textit{Email address: }\texttt{i.leader@dpmms.cam.ac.uk}

\medskip

\noindent John~M.~Mackay, \textsc{School of Mathematics, University of Bristol, The Fry Building, Woodland Road, Bristol, BS8 1UG, UK.}\par\noindent\nopagebreak\textit{Email address: }\texttt{john.mackay@bristol.ac.uk}}
\pagestyle{fancy}
\fancyhf{}
\fancyhead [LE, RO] {\thepage}
\fancyhead [CE] {DAVID ELLIS, MARIA-ROMINA IVAN, IMRE LEADER, AND JOHN M. MACKAY}
\fancyhead [CO] {ANTIPODAL PATHS IN COVERS OF SPHERES}
\renewcommand{\headrulewidth}{0pt}
\renewcommand{\l}{\rule{6em}{1pt}\ }
\title{\Large{\textbf{ANTIPODAL PATHS IN COVERS OF SPHERES}}}
\author{\small{DAVID ELLIS, MARIA-ROMINA IVAN, IMRE LEADER, AND JOHN M. MACKAY}}
\date{ }
\maketitle

\begin{abstract}
In this note we show that if the sphere $\Sph^n$ is covered by $k$ open sets with $n \geq 2k-2$, then one of these sets contains a path with antipodal endpoints. This is best possible in the sense that the statement fails for $n < 2k-2$. The result can be seen as a spherical analogue of a well-known conjecture of Norine on edge-colourings of the discrete hypercube.
\end{abstract}

\section{Background}
The Lusternik--Schnirelmann theorem \cite{lusternik1930methodes}, one of the equivalent formulations of the Borsuk--Ulam theorem, states that if $\Sph^n$ is covered by $n+1$ open sets $U_0, \ldots, U_n$, then there exists an antipodal pair of points $x, -x \in \Sph^n$ lying in some $U_i$.

There are many generalisations and variants of this theorem (see e.g. \cite{Matousek-03-Borsuk-Ulam,Steinlein-85-Borsuk-survey}), but we have been unable to find the following variant stated explicitly in the literature (and, as it is a natural spherical analogue of a well-known combinatorial conjecture, as we discuss later, we thought it worth recording). 

We say a set $U \subset \Sph^n$ contains an \emph{antipodally connecting path} if there exists a path in $U$ connecting two antipodal points. We show the following.
\begin{theorem}
\label{thm:main}
Let $n \geq 2k-2$. Then, for every cover of $\Sph^n$ by $k$ open sets, one of the sets contains an antipodally connecting path. The bound $n \geq 2k-2$ is sharp.
\end{theorem}

It turns out that this follows from some (fairly) classical results in topology, together with two short additional arguments. We remark that a weaker sufficient bound (of $n \geq 2k-1$) also follows quickly from \cite[Theorem 3.2]{Yang-55-Borsuk-Ulam-II}.

For some intuition on the above theorem, suppose that we can cover $S^n$ by $k$ open sets, none of which contains an antipodally connecting path. We might then hope to split each of our open sets into two open sets, neither containing an antipodal pair -- or, at least, to cover each of our open sets by two such open sets. This would yield a cover by $2k$ open sets, none containing an antipodal pair, and the Lusternik-Schnirelmann theorem would then give $2k \geq n+2$. 

This is of course not a valid argument, as there is no reason why we can cover each of our open sets by two open sets not containing an antipodal pair. For example, in $S^1$, consider the open set consisting of the whole circle with the three points of an equilateral triangle removed.

Note also that certainly there are covers of $\Sph^n$ by $n+1$ open sets without any set containing an antipodally connecting path. Indeed, as $\Sph^n$ has topological dimension $n$, one can cover $\Sph^n$ by $n+1$ open sets whose connected components have diameter $<\epsilon$ for any positive $\epsilon$. (Recall that, as shown by Ostrand, the topological dimension of a metric space $X$ is the  
minimal integer $d$ such that for every open cover $\{U_\alpha\}$ of $X$, there exist families of open sets $\mathcal{F}_i = \{V_{i,\alpha}\}$ for $1 \leq i \leq d+1$ such that $V_{i,\alpha} \subset U_{\alpha}$ for all $i$ and all $\alpha$, $\cup_{i=1}^{d+1}\mathcal{F}_i$ is a cover of $X$, and $V_{i,\alpha} \cap V_{i,\alpha'} = \emptyset$ for all $i$ and all $\alpha\neq \alpha'$. To obtain the above statement we simply apply this to the open cover $\{U_\alpha\}$ consisting of the $(\epsilon/2)$-balls, and we take the $i$th open set with all connected components of diameter $< \epsilon$ to be $\cup_{\alpha}V_{i,\alpha}$, for $1 \leq i \leq n+1$.)

\section{Proof of Theorem~\ref{thm:main}}
We say a map $f:\Sph^n \to Y$ has an {\em antipodal coincidence} if there exists $x \in \Sph^n$ such that $f(x) = f(-x)$.

As remarked above, $\Sph^n$ can be covered by $n+1$ open sets without any set containing an antipodally connecting path. To get to open covers with about half as many sets, we apply the same argument to the image of a continuous map $f:\Sph^n\to Y$ for some $(k-1)$-dimensional metric space $Y$, where $f$ has no antipodal coincidences.

The following theorem characterises for which dimensions such maps exist. The sharpness part, which comes from an explicit construction of maps without antipodal coincidences, is required to prove the sharpness part of Theorem~1.
\begin{theorem}[Hopf \cite{hopf}, \v S\v cepin \cite{Shchepin-74}]
\label{thm:covers}
Every continuous map $f:\Sph^n \to Y$ to a compact metric space of topological dimension $m$, where $n \geq 2m$, must have an antipodal coincidence. The bound $n \geq 2m$ is sharp.
\end{theorem}
See \cite[Theorem 1]{Volov-Shch-05-antipodal} and the discussion and references therein for the somewhat complicated history of the (re)discovery of this and related results.

\begin{proof}
    [Proof of Theorem~\ref{thm:main}]
    Suppose $n \geq 2k-2$, and $\Sph^n$ is covered by $k$ open sets.  Consider the cover of $\Sph^n$ by the connected components of this cover, and take a finite subcover $\cU$; observe that this subcover has multiplicity at most $k$.
    Choose a continuous partition of unity $\{f_U\}_{U \in \cU}$ subordinate to this cover $\cU$, i.e., each $f_U:\Sph^n \to \mathbb{R}_{\geq 0}$ is a continuous, non-negative real-valued function with support $U$, and for all $x \in \Sph^n$ we have $\sum_{U \in \cU} f_U(x)=1$.

    Consider the continuous map $F:\Sph^n \to \Delta \subset \R^{\cU}, F(x) = (f_U(x))_{U \in \cU}$, where $\Delta$ is the standard simplex in $\R^{\cU}$. Since $\cU$ has multiplicity at most $k$, $F$ maps into the $(k-1)$-skeleton of $\Delta$, which has topological dimension $k-1$. Thus by Theorem~\ref{thm:covers} there exists $x \in \Sph^n$ with $F(x)=F(-x)$. Thus for some $U \in \cU$ we have $f_U(x)=f_U(-x)\neq 0$, so $x$ and $-x$ are both in the same open connected set $U$.

    Conversely, suppose $n < 2k-2$.
    Then by Theorem~\ref{thm:covers} there exists a compact metric space $Y$ of topological dimension $k-1$ and a continuous map $f:\Sph^n \to Y$ so that for all $x \in \Sph^n$, $f(x)\neq f(-x)$.  Let $\epsilon = \min_{x \in \Sph^n} d_Y(f(x),f(-x))$.  As $\Sph^n$ is compact, we have $\epsilon >0$.  As $Y$ has topological dimension at most $k-1$, there exists an open cover $V_1, \ldots, V_k$ of $Y$ such that every connected component of each $V_i$ has diameter less than $\epsilon$. Let $U_i = f^{-1}(V_i)$ for each $i=1,\ldots, k$. Then clearly $U_1, \ldots, U_k$ is an open cover of $\Sph^n$ with no antipodally connecting path in any $U_i$.
\end{proof}
We remark that we were inspired to consider this problem as it is a spherical analogue of a now notorious conjecture in combinatorics (originally due to Norine \cite{norine} and recast in a slightly different form by Feder and Subi \cite{feder}) that in any two-colouring of the edges\footnote{As usual, the {\em edges} of the $n$-dimensional discrete hypercube are the $n2^{n-1}$ line segments between vertices that differ in exactly one coordinate.} of the $n$-dimensional discrete hypercube $\{-1,1\}^n$, there exists a path between two antipodal vertices with at most one colour-change. This problem remains wide open, though recently Hollom \cite{hollom} proved that there is always a geodesic (no direction used twice), and therefore a path, between two antipodal vertices with at most $O(\sqrt{n})$ colour-changes. Until Hollom's result, all known upper bounds were linear in $n$.
\vspace{-1em}
\bibliographystyle{amsplain}
\bibliography{bibliography}
\vspace{-1em}
\Addresses
\end{document}